\documentclass[12pt,a4paper]{article}
\usepackage[cp1251]{inputenc}
\usepackage{amsmath,indentfirst,amsfonts,amssymb,amsthm}
\usepackage[english]{babel}
\newtheorem{theorem}{Theorem}
\newtheorem*{remark}{Remark} 
\newcommand{\Cbb}{\mathbb{R}}
\newcommand{\Lcdots}{\genfrac{}{}{0pt}{0}{}{ \cdots}}
\newcommand{\Plus}{\genfrac{}{}{0pt}{0}{ } {+}\;}
\newcommand{\Rc}{\mathcal{R}}
\newcommand{\Zc}{\mathbf{R}}
\begin{document}
\pagestyle{empty}

\noindent
\begin{center}
\textbf{\Large Interpolation of functional by integral continued C-fractions}\\
Volodymyr L. Makarov\\ 
Institute of Mathematics of the NAS of  Ukraine, Kiev, Ukraine\\ 
\textit{E-mail:} makarovimath@gmail.com\\
and \\
Mykhaylo M. Pahirya\\
State University of Mukachevo, Ukraine \\ 
\textit{E-mail:} pahirya@gmail.com\\ 

January  22, 2018 
\end{center} 
\noindent
\textbf{Abstract.}
The functional interpolation problem on a  continual  set of nodes by an integral continued  C-fraction is studied. The necessary and sufficient conditions for its solvability are found. As a particular case, the considered integral continued  fraction contains a standard interpolation continued C-fraction which is used to approximate the functions of one variable.
 


\section{Introduction} 
Recently, the considerable amount of scientific studies were devoted to the generalization of  interpolation theory of  functions of real (complex) variable to the case of functionals and operators in abstract spaces, see the monographs \cite{MakHlob, MakHlobYan} for example.
Continued fractions \cite {Jons} and branched continued fractions introduced by Skorobohat'ko V.Ya. \cite{Skor} were generalized by integral continued fractions which proposed by Syvavko M.S. \cite{Syv}.
The problem of interpolation with integral continued  fractions was first considered in the article  \cite{Myh},
further expansions and generalizations of this work are contained in the paper \cite {MakHlobMyh}. 
 
 Another class of interpolation integral continued fractions has been investigated in the paper \cite{MakDem1}. 
 This class differs from the previously studied integral continued fractions by the fact that \(n\)--s floor of fraction containing \(n\)--tuple integral.
 Interpolation integral operator continued fractions in Banach spaces were investigated in the article \cite{MakHlobDem}.
 The natural  generalization of the classical Thiele continued fraction  to the interpolation integral Thiele--type continued fractions were proposed in \cite{MakDemk, MakDem2, MakDem4}.
 
 The purpose of this work is the study of interpolation of a functional given on the continual set of nodes by the integral C--type continued fractions. 
 Such integral continued fractions contains the interpolation continued  C--fraction as a particular case, so it is a generalization of one of the types of continued fraction used for interpolation of functions \cite{Pah16_8}.
  
\section{Statemet of the problem} 
Assume that \(x(z),x_i(z) \in C[0,1], i=\overline{0,n},\) are some given functions,  \( x_i(z)\neq x_j(z)\). 
Let   \(F(x(\cdot))\) be a certain functional defined in the space of piecewise continuous functions \(Q [0,1]\). 
First we define continual nodes
\begin{equation}\label{MP0}
x^0(z)=x_0(z),\quad x^i(z,\xi)=x_0(z)+H(z-\xi)(x_{i}(z)-x_0(z)),\qquad i=\overline{1,n}. 
\end{equation}
where  \(0\leq \xi\leq 1,\) and \(H(\cdot)\) is the Heaviside function.
 
Let  \(\mathbf{b}_0,\mathbf{a}_i,\mathbf{b}_i,i=\overline{1,n},\) are  numbers, functions, operators, functionals etc. We denote a finite continued fraction  \cite{Jons}
\[
D_n=\mathbf{b}_0+\cfrac{\mathbf{a}_1}{\mathbf{b}_1+\cfrac{\mathbf{a}_2}{\mathbf{b}_2+\genfrac{}{}{0pt}{0}{}{\ddots\genfrac{}{}{0pt}{0}{}{+\cfrac{\mathbf{a}_n}{\mathbf{b}_n}}}}},
\]
as 
\[
D_n=\mathbf{b}_0+\cfrac{\mathbf{a}_1}{\mathbf{b}_1}\Plus \cfrac{\mathbf{a}_2}{\mathbf{b}_2}\Plus\Lcdots \Plus \cfrac{\mathbf{a}_n}{\mathbf{b}_n}.
\] 

Let \(\Rc \subset \Cbb \) be a compact,
\(\Zc = \{x_i \,: \, x_i \in \Rc, x_i \ne x_j, \, i, j = \overline {0, n} \} \) and the function 
\(f \in \mathbf{C} (\Rc) \) is defined by its values at the points of the set \(\Zc \), \(y_i = f (x_i), \, i = \overline{0,n}.\)  The function \(f \) can be interpolated by the set of values \(\{y_i \} \) in different ways, for example, polynomials, splines and continued fractions. A problem of function interpolation   by interpolation continued  C--fraction (C--ICF) 
\begin{equation}\label{MP16}
D^{(c)}_n(x)=a^{(c)}_0+\frac{a^{(c)}_1(x-x_{0})}{1}\Plus \frac{a^{(c)}_2(x-x_{1})}{1}\Plus \Lcdots \Plus \frac{a^{(c)}_n(x-x_{n-1})}{1}
\end{equation}
investigated in monograph  \cite{Pah16_8}.
The coefficients of C--ICF are determined through the interpolation nodes \(\Zc \) and the set of values of the function \(\{y_i \} \) by means of a finite-continued fraction recurrence
\begin{equation}\label{MP17}
\begin{split}
a^{(c)}_k&=\frac{1}{x_k-x_{k-1}}\left(-1+\frac{a^{(c)}_{k-1}(x_k-x_{k-2})}{-1}\Plus\!  
\Lcdots\! \Plus \frac{a^{(c)}_{2}(x_k-x_{1})}{-1}\Plus  \frac{a^{(c)}_1(x_k-x_0)}{y_k-y_0}\right),  k=\overline{2,n}, \\
a^{(c)}_0&=y_0, \quad a^{(c)}_1=\frac{y_1-y_0}{x_1-x_0}.
\end{split}
\end{equation}
If all the interpolation nodes \(x_i, i = \overline {0,n}, \) tend  to the same value \(x_ * \in \Rc \), then C--ICF \eqref{MP16}
transforms into a regular continued  C--fraction that will correspond to the power series for \(f \) around the 
point  \(x _* \).

Consider the set of integral continued fractions (ICF) of the form
\[
Q_n(x(\cdot),\xi)=a_0+\cfrac{\int\limits_0^1 a_1(z_1)[x(z_1)-x_0(z_1)] dz_1}{1}\Plus \cfrac{\int\limits_0^1 a_2(z_2)[x(z_2)-x^1(z_2,\xi)]  dz_2}{1}\Plus\Lcdots \Plus 
\]
\begin{equation}\label{MP1}
\Plus \cfrac{\int\limits_0^1 a_{n-1}(z_{n-1})[x(z_{n-1})-x^{n-2}(z_{n-1},\xi)] dz_{n-1}}{1}\Plus \cfrac{\int\limits_0^1 a_n(z_n)[x(z_n)-x^{n-1}(z_n,\xi)] dz_n}{1}\, ,
\end{equation}
where  \(a_0,a_1(z),a_2(z),\dots,a_n(z)\) are certain kernels. 

We formulate an interpolation problem as follows:
Inside the set of ICF defined by \eqref {MP1} find a integral continued fraction which  satisfies the interpolation conditions at continual nodes \eqref{MP0}
\begin{equation}\label{MP2}
F(x_0(\cdot))=Q_n(x_0(\cdot),\xi), \quad
F(x^i(\cdot,\xi))=Q_n(x^i(\cdot,\xi),\xi), \qquad i=\overline{1,n}, \quad \forall\, \xi \in [0,1]. 
\end{equation}
Such ICF will include C--ICF  \eqref{MP16} as a partial case.
 Such an ICF is called an integral interpolation continued C--fraction (C--IICF).
  
\section{Interpolation of functional by integral interpolation continued  C--fraction} 
Define the kernels \(a_0, a_1 (z), \dots, a_n (z) \) from the condition that C--IICFL \eqref {MP1} satisfies \eqref {MP2}. Note that from \eqref {MP1}, \eqref {MP2} you can directly get the expression
\[
Q_n(x^k(\cdot,\xi),\xi)=a_0+\cfrac{\int\limits_0^1 a_1(z_1)[x^k(z_1,\xi)-x_0(z_1)] dz_1}{1}\Plus \cfrac{\int\limits_0^1 a_2(z_2)[x^k(z_2,\xi)-x^1(z_2,\xi)]  dz_2}{1}\Plus\Lcdots \Plus \]
\begin{equation}\label{MP3}
\Plus \cfrac{\int\limits_0^1 a_k(z_{k})[x^k(z_{k},\xi)-x^{k-1}(z_{k},\xi)] dz_{k}}{1}\, , \qquad k=\overline{0,n}.
\end{equation}

\begin{theorem} 
	Let the functional \(F(x(\cdot)) \) be  \((n-1)\)--times Gateaux differentiable and the following formulas are valid.
	In order that C--IICF \eqref{MP1} to satisfy the continual interpolation conditions \eqref{MP2} it is necessary that its kernels are determined by the formulas
\begin{equation}\label{MP4}
\begin{split}
&a_k(\xi)=\frac{-1}{x_k(\xi)-x_{k-1}(\xi)} \frac{d 
	}{d \xi} \Bigg( \cfrac{\int\limits_0^1 a_{k-1}(z_{k-1})[x^k(z_{k-1},\xi)-x^{k-2}(z_{k-1},\xi)] d z_{k-1}}{-1}\Plus  \\
&\Plus \cfrac{\int\limits_0^1 a_{k-2}(z_{k-2})[x^k(z_{k-2},\xi)\!-\!x^{k-3}(z_{k-2},\xi)] d z_{k-2}}{-1}\Plus\!
\Lcdots\! \Plus
\cfrac{\int\limits_0^1 a_{2}(z_2)[x^k(z_2,\xi)\!-\!x^{1}(z_2,\xi)] d z_2}{-1}\!\Plus
 \\
& \Plus
\cfrac{\int\limits_0^1 a_1(z_1)[x^k(z_1,\xi)-x_0(z_1)] d z_1}{F(x^k(\cdot,\xi))-F(x_0(\cdot))}\Bigg) ,\qquad k=\overline{2,n},
\\
&a_0= F(x_0(\cdot)), \; a_1(\xi)=\frac{-1}{x_1(\xi)-x_0(\xi)} \, \frac{d}{d \xi} F(x^1(\cdot,\xi)).
\end{split}
\end{equation}
\end{theorem}
\begin{proof}
	For the case \(k = 0,1 \) the formulas are obvious.
	When \(k = m \), from \eqref {MP2}, \eqref {MP3} we get
\[
F(x^m(\cdot,\xi))=a_0+\cfrac{\int\limits_0^1 a_1(z_1)[x^m(z_1,\xi)-x_0(z_1)] dz_1}{1}\Plus \cfrac{\int\limits_0^1 a_2(z_2)[x^m(z_2,\xi)-x^1(z_2,\xi)]  dz_2}{1}\Plus \Lcdots\Plus 
\]
\[
\Plus \cfrac{\int\limits_0^1 a_{m-1}(z_{m-1})[x^m(z_{m-1},\xi)-x^{m-2}(z_{m-1},\xi)] dz_{m-1}}{1}
\Plus \cfrac{\int\limits_0^1 a_m(z_{m})[x^m(z_m,\xi)-x^{m-1}(z_m,\xi)] dz_m}{1}.
\]
By sequentially inverting the continued fraction, we get
\[
\int\limits_{\xi}^1 a_m(z_m)[x^m(z_m,\xi)-x^{m-1}(z_m,\xi)] dz_m+1 \!=\!\cfrac{\int\limits_0^1 a_{m-1}(z_{m-1})[x^m(z_{m-1},\xi)\!-\!x^{m-2}(z_{m-1},\xi)] dz_{m-1}}{-1}\Plus
\]
\[
\Plus \Lcdots\Plus\cfrac{\int\limits_0^1 a_2(z_2)[x^m(z_2,\xi)-x^1(z_2,\xi)]  dz_2}{-1} \Plus \cfrac{\int\limits_0^1 a_1(z_1)[x^m(z_1,\xi)-x_0(z_1)] dz_1}{F(x^m(\cdot,\xi))-F(x_0(\cdot))}. 
\]
The subsequent differentiation of both parts of this relation with respect to the variable \(\xi\) yields formula
 \eqref{MP4}. 
\end{proof}
\begin{theorem}\label{MPTh2}
		Let 
	\begin{equation}\label{MP19}
	F(x(\cdot))=f\big(\int_0^1 x(t) dt\big)
	\end{equation}	
and kernels 	 C--IICF \eqref{MP1} are determined by the formulas \eqref{MP4}. For C--IICF  
to satisfy the interpolation conditions \eqref{MP2} it is sufficient that the function
 \(f(s)\in C^{(n-1)}(-\infty,+\infty)\).
\end{theorem}
\begin{proof}
Assume that the kernels are determined by the formulas \eqref{MP4}. From \eqref{MP3}  it follows directly that  
\(Q_n(x_0(\cdot),\xi)=F(x_0(\cdot)), \, Q_n(x^1(\cdot,\xi),\xi)=F(x^1(\cdot,\xi)), \; \forall \xi \in [0,1],\)
for   \(k=0,1\).  
When \(k=2\) we obtain  from  \eqref{MP3} the formula
 \begin{equation}\label{MP12}
 Q_n(x^2(\cdot,\xi),\xi)=a_0+\cfrac{\int\limits_0^1 a_1(z_1)[x^2(z_1,\xi)-x_0(z_1)]d z_1}{1+\int\limits_0^1 a_2(z_2)[x^2(z_2,\xi)-x^1(z_2,\xi)]d z_2}.
 \end{equation}
 Next, we substitute value of the kernel  \(a_2 (z_2)\) and find the fraction's denominator
  \[
 P_2=1+\int\limits_0^1 a_2(z_2)[x^2(z_2,\xi)-x^1(z_2,\xi)]d z_2=
 1-\int\limits_\xi^1   
 \frac{d}{d z_2} \Bigg(\cfrac{\int\limits_{z_2}^1  a_1(z_1)[x_2(z_1)-x_0(z_1)]d z_1}{F(x^2(\cdot,\xi))-F(x_0(\cdot))}\Bigg)\, d z_2=
 \]
 \[
 =1-\lim_{z_2 \to 1} K_2[z_2,x_2]+\cfrac{\int\limits_\xi^1  a_1(z_1)[x_2(z_1)-x_0(z_1)]d z_1}{F(x^2(\cdot,\xi))-F(x_0(\cdot))}, 
 \] 
 where
 \[ 
 K_2[z_2,x_2]= \cfrac{\int\limits_{z_2}^1  a_1(z_1)[x_2(z_1)-x_0(z_1)]d z_1}{F(x^2(\cdot,z_2))-F(x_0(\cdot))}.
 \]
 By applying the  L'Hospital rule and considering \eqref{MP4}, \eqref {MP19}, we get
  \[
\lim_{z_2 \to 1}  K_2[z_2,x_2]=-\lim_{z_2 \to 1}\cfrac{ a_1(z_2)[x_2(z_2)-x_0(z_2)]}{\cfrac{d F(x^2(\cdot,z_2))}{d z_2}} =
\lim_{z_2\to 1}\cfrac{a_1(z_2)}{a_1(z_2)\big|_{x_1(z_2)\to x_2(z_2)}}=
\]
\begin{equation}\label{MP15}
 =\lim_{z_2 \to 1} \cfrac{x_2(z_2)-x_0(z_2)}{x_1(z_2)-x_0(z_2)}\,
 \cfrac{f'\Big(\int\limits_0^1 x^1(t,z_2)dt\Big)(x_0(z_2)-x_1(z_2))} {f'\Big(\int\limits_0^1 x^2(t,z_2))dt\Big)(x_0(z_2)-x_2(z_2)) }=1, \qquad \forall\, x_2(z_2).
 \end{equation}
  We substitute the calculated value of the limit  \(K_2 [z_2, x_2] \) into the expression for \(P_2 \), and then substitute the result into  \eqref {MP12}.
 Consequently, we obtain \[(Q_n(x^2(\cdot,\xi),\xi)=F(x^2(\cdot,\xi)).\] 
 Let \(k=3\). From \eqref{MP3}  we get 
\[
Q_n(x^3(\cdot,\xi),\xi)=F(x_0(\cdot))+\cfrac{\int\limits_0^1 a_1(z_1)[x^3(z_1,\xi)-x_0(z_1)] dz_1}
{1}\Plus  
\]
\[
\Plus\cfrac{\int\limits_0^1 a_2(z_2)[x^3(z_2,\xi)-x^1(z_2,\xi)] dz_2}{1}\Plus\cfrac{\int\limits_0^1 a_3(z_3)[x^3(z_3,\xi)-x^2(z_3,\xi)] dz_3}{1}.
\]
We substitute the value of the kernel \(a_3(z_3)\) and evaluate the last denominator to get
\[
P_3=1+\int\limits_0^1 a_3(z_3)[x^3(z_3,\xi)-x^2(z_3,\xi)]dz_3= 1-\int\limits_\xi^1 \frac{d}{d z_3} \left\{
\cfrac{\int\limits_0^1 a_2(z_2)[x^3(z_2,z_3)-x^1(z_2,z_3)]d z_2}{-1}\Plus  \right. 
\]
\[
\left.
\Plus\cfrac{\int\limits_0^1 a_1(z_1)[x^3(z_1,z_3)-x_0(z_1)]d z_1}{F(x^3(\cdot,z_3))-F(x_0(\cdot))} \right\}
dz_3=1-\lim_{z_3 \to 1} K_3[z_3,x_3]+\cfrac{\int\limits_{\xi}^1 a_2(z_2)[x_3(z_2)-x_1(z_2)]d z_2}{-1} \Plus
\]
\[
\Plus\cfrac{\int\limits_{\xi}^1 a_1(z_1)[x_3(z_1)-x_0(z_1)]d z_1}{F(x^3(\cdot,z_3))-F(x_0(\cdot))},
\]
where 
\[
 K_3[z_3,x_3]=\frac{\int\limits_{z_3}^1 a_2(z_2)[x_3(z_2)-x_1(z_2)]d z_2}{-1} \Plus \cfrac{\int\limits_{z_3}^1 a_1(z_1)[x_3(z_1)-x_0(z_1)]d z_1}{F(x^3(\cdot,z_3))-F(x_0(\cdot))}.  
\]
By applying the L'Hospital rule, taking into account 
\eqref {MP15}, we obtain
\[
 \lim_{z_3 \to 1} K_3[z_3,x_3]= \lim_{z_3 \to 1} 
 \cfrac{a_2(z_3)[x_3(z_3)-x_1(z_3)]}{\cfrac{a_1(z_3)[x_3(z_3)-x_ 0(z_3)]}{F(x^3(\cdot,z_3))-F(x_0(\cdot))}+\cfrac{\int\limits_{z_3}^1 a_1(z_1)[x_3(z_1)-x_0(z_1)]d z_1}{(F(x^3(\cdot,z_3))-F(x_0(\cdot)))^2}\, \cfrac{d}{d z_3} F(x^3(\cdot,z_3))}= 
\] 

\[
=\lim_{z_3 \to 1} \cfrac{x_3(z_3)\!-\!x_1(z_3)}{x_2(z_3)\!-\!x_1(z_3)}\,
\cfrac{a_1(z_3)[x_2(z_3)-x_0(z_3)]\!+\!\cfrac{d}{d z_3}F(x^2(\cdot,z_3))}{a_1(z_3)[x_3(z_3)\!-\!x_0(z_3)] +\cfrac{d}{d z_3}F(x^3(\cdot,z_3))}\, \cfrac{F(x^3(\cdot,z_3))-F(x_0(\cdot))}{F(x^2(\cdot,z_3))-F(x_0(\cdot))}=
\]
\[
=\lim_{z_3 \to 1} \cfrac{x_3(z_3)-x_1(z_3)}{x_2(z_3)-x_1(z_3)}
\,\cfrac{\cfrac{d}{d z_3} F(x^3(\cdot,z_3))}{\cfrac{d}{d z_3} F(x^2(\cdot,z_3))} \times  \]
\[
\times \cfrac{[x_2(z_3)-x_0(z_3)]\cfrac{d}{d z_3} F(x^1(\cdot,z_3))-[x_1(z_3)-x_0(z_3)]\cfrac{d}{d z_3} F(x^2(\cdot,z_3))}{[x_3(z_3)-x_0(z_3)]\cfrac{d}{d z_3} F(x^1(\cdot,z_3))-[x_1(z_3)-x_0(z_3)]\cfrac{d}{d z_3} F(x^3(\cdot,z_3))}.  
\]
Using \eqref {MP19} and the mean--value theorem, we obtain
\[
\lim_{z_3 \to 1}K_3[z_3,x_3]=\lim_{z_3 \to 1}  \cfrac{x_3(z_3)-x_1(z_3)}{x_2(z_3)-x_1(z_3)}
\, \cfrac{f''(\theta_1)\Big(\int\limits_0^1 x^1(s,z_3)ds -\int\limits_0^1 x^2(s,z_3)ds\Big)}{f''(\theta_2)\Big(\int\limits_0^1 x^1(s,z_3)ds -\int\limits_0^1 x^3(s,z_3)ds\Big)}=
\]
\[
=\lim_{z_3 \to 1}\cfrac{a_2(z_3)}{a_2(z_3)\big|_{x_2(z_3)\to x_3(z_3)}}=1, \qquad \forall\, x_3(z_3),
\]
where 
\[
\theta_1=\int\limits_0^1 x^1(s,z_3)ds+\tau_1 \Big(\int\limits_0^1 x^2(s,z_3)ds-\int\limits_0^1 x^1(s,z_3)ds\Big), \]
\[
\theta_2=\int\limits_0^1 x^1(s,z_3)ds+\tau_2 \Big(\int\limits_0^1 x^3(s,z_3)dsm-\int\limits_0^1 x^1(s,z_3)ds\Big), \qquad \tau_1,\tau_2 \in (0,1).
\]
Using equation\eqref{MP3} in the general case for \(k = m \) we obtain
\[
Q_n(x^m(\cdot,\xi),\xi)=F(x_0(\cdot))+\frac{\int\limits_0^1 a_1(z_1)[x^m(z_1,\xi)\!-\!x_0(z_1)]d z_1}{1}\Plus \frac{\int\limits_0^1 a_2(z_2)[x^m(z_2,\xi)\!-\!x^1(z_2,\xi)]d z_2}{1}\Plus
\]
\begin{equation}\label{MP11}
\Plus \Lcdots \Plus \frac{\int\limits_0^1 a_{m-1}(z_{m-1})[x^m(z_{m-1},\xi)-x^{m-2}(z_{m-1},\xi)]d z_{m-1}}{1+\int\limits_0^1 a_m(z_m)[x^m(z_m,\xi)-x^{m-1}(z_m,\xi)]d z_m}\, .
\end{equation}
Which permits us to find the value of the last denominator
\[
P_m=1+\int\limits_0^1 a_m(z_m)[x^m(z_m,\xi)-x^{m-1}(z_m,\xi)]d z_m.
 \]
To do that we use the values of \(a_m (z_m)\) from \eqref{MP4}. Then 
\[
P_m=1-\int\limits_\xi^1 \frac{d}{d z_m}\left( \cfrac{\int\limits_0^1 a_{m-1}(z_{m-1})[x^m(z_{m-1},z_m)\!-\!x^{m-2}(z_{m-1},z_m)] dz_{m-1}}{-1}\Plus\right.
\]
\[
\left.
\Plus \Lcdots\Plus\cfrac{\int\limits_0^1 a_2(z_2)[x^m(z_2,z_m)-x^1(z_2,z_m)]  dz_2}{-1} \Plus \cfrac{\int\limits_0^1 a_1(z_1)[x^m(z_1,z_m)-x_0(z_1)] dz_1}{F(x^m(\cdot,z_m))-F(x_0(\cdot))}\right) d z_m=
\]
\[
=1-\lim_{z_m \to 1}
K_m[z_m,x_m]+  
\cfrac{\int\limits_\xi^1 a_{m-1}(z_{m-1})[x_m(z_{m-1})\!-\!x_{m-2}(z_{m-1})] dz_{m-1}}{-1}\Plus
\]
\[
\Plus \Lcdots\Plus\cfrac{\int\limits_\xi^1 a_2(z_2)[x_m(z_2)-x_1(z_2)]  dz_2}{-1} \Plus \cfrac{\int\limits_\xi^1 a_1(z_1)[x_m(z_1)-x_0(z_1)] dz_1}{F(x^m(\cdot,\xi))-F(x_0(\cdot))}\, ,
\]
where
\[
K_m[z_m,x_m]=
  \cfrac{\int\limits_{z_m}^1 a_{m-1}(z_{m-1})[x_m(z_{m-1})-x_{m-2}(z_{m-1})] dz_{m-1}}{-1}\Plus\Lcdots\Plus  
\]
\[
 \Plus \cfrac{\int\limits_{z_m}^1 a_2(z_2)[x_m(z_2)\!-\!x_1(z_2)]  dz_2}{-1} \Plus\cfrac{\int\limits_{z_m}^1 a_1(z_1)[x_m(z_1)-x_0(z_1)] dz_1}{F(x^m(\cdot,z_m))-F(x_0(\cdot))} . 
\]
Similarly, one can prove that
\[ 
\lim\limits_{z_m \to 1} K_m[z_m,x_m]=\lim_{z_m\to 1}\cfrac{a_{m-1}(z_m)}{a_{m-1}(z_m)\big|_{x_{m-1}(z_m)\to x_m(z_m)}}=1, \qquad \forall\, x_m(z_m).
\] 
By substituting the obtained value of \(P_m \) into \eqref {MP11} we get the sought for result, namely
 \[
 Q_n(x^m(\cdot,\xi),\xi)=F(x^m(\cdot,\xi)).
 \] 
 For brevity we omit the case \(k > 3\), which can be proved similarly.
\end{proof}
\begin{remark}
	Conditions \eqref{MP19} formulation of Theorem 
	\ref{MPTh2} can be changed to a more general, when the results of \cite{AvSm,MakHlobKashMyh} are employed.
\end{remark}
\begin{theorem}\label{MPTh3}
Assume the conditions of Theorem \ref{MPTh2}  are valid and  \(\xi=0, x(z)\equiv x, x_i(z)\equiv x_i, i=\overline{0,n}.\) Then  C--IICF \eqref{MP1} 
will coincide with the  C--ICF  \eqref{MP16}.	
\end{theorem}
\begin{proof}
Taking into account the conditions of the theorem 
\ref{MPTh3}, we obtain the following form of a continued fraction \eqref{MP1}  
\begin{equation}\label{MP8}
Q_n(x)=a_0+\cfrac{(x-x_0)\int\limits_0^1 a_1(z_1)dz_1}{1}\Plus \cfrac{(x-x_1)\int\limits_0^1 a_2(z_2) dz_2}{1}\Plus\Lcdots \Plus \cfrac{(x-x_{n-1})\int\limits_0^1 a_n(z_n) dz_n}{1}\, .
\end{equation}
From the interpolation condition \eqref{MP2} we get
\[
a_0\!=\!f(x_0),\,  \tilde{a}_1\!=\!\int\limits_0^1 a_1(z_1) dz_1 =\frac{f(x_1)\!-\!f(x_0)}{x_1-x_0},\, \tilde{a}_2\!=\!\int\limits_0^1 a_2(z_2) dz_2=\frac{1}{x_2\!-\!x_1}\bigg(-1+\frac{\tilde{a}_1(x_2-x_0)}{f(x_2)\!-\!f(x_0)}\bigg).
\]
Using the mathematical induction we intend to show that the formula 
\[
\tilde{a}_m=\int\limits_0^1 a_m(z_m) dz_m= \frac{1}{x_m-x_{m-1}}\bigg(-1+\frac{\tilde{a}_{m-1}(x_m-x_{m-2})}{-1}\Plus\frac{\tilde{a}_{m-2}(x_m-x_{m-3})}{-1}\Plus 
\]
\begin{equation}\label{MP9}
\Plus\Lcdots \Plus\frac{\tilde{a}_{2}(x_m-x_{1})}{-1}
\Plus\frac{\tilde{a}_{1}(x_m-x_{0})}{f(x_m)-f(x_0)}
\bigg)\, 
\end{equation}
holds true for \(m = \overline {3, n}\).

For \(m = 1,2\) formula \eqref {MP9} is obviously valid.
Let's assume that it is valid for \(m  = \overline {1, k-1}. \) Then for \( m = k \) the continued fraction \eqref {MP8} can be written as
\begin{equation}\label{MP10}
f(x_k)=a_0+\cfrac{(x_k-x_0)\tilde{a}_1}{1}\Plus 
\Lcdots \Plus\cfrac{(x_k-x_{k-2})\tilde{a}_{k-1}}{1}\Plus \cfrac{(x_k-x_{k-1})\tilde{a}_k(z_k)}{1}\, .
\end{equation}
By inverting the continued fraction \eqref{MP10}, we get
\[
\tilde{a}_k=\int\limits_0^1 a_k(z_k) dz_k= \frac{1}{x_k-x_{k-1}}\bigg(-1+
\frac{\tilde{a}_{k-1}(x_k-x_{k-2})}{-1}
\Plus\frac{\tilde{a}_{k-2}(x_k-x_{k-3})}{-1}
\Plus\Lcdots \Plus
\]
\begin{equation}\label{MP21}
\Plus \frac{\tilde{a}_{2}(x_k-x_{1})}{-1}
\Plus\frac{\tilde{a}_{1}(x_k-x_{0})}{f(x_k)-f(x_0)}
\bigg).
\end{equation}
By its form, the right-hand side of \eqref {MP21} coincides with the right-hand side of \eqref {MP17} albeit with different notation, in addition to that, the initial conditions for both formulas are the same.
Consequently, formula \eqref {MP9} is the same as formula \eqref {MP17}.
\end{proof}

\end{document}